\documentclass[12pt]{article}
\usepackage{amsmath}
\usepackage{amsthm,amssymb,amsfonts}
\usepackage{array}
\usepackage{amsmath}
\usepackage{latexsym}
\usepackage{amsfonts}
\usepackage{amssymb}
\usepackage{amscd}
\newtheorem{thm}{Theorem}[section]
\newtheorem{proposition}{Proposition}
\newtheorem{lem}[thm]{Lemma}
\newtheorem{corollary}[thm]{Corollary}
\newtheorem{example}{Example}
\newtheorem{rmk}{Remark}

\newcommand{\N}{\mathbb{N}} 
\newcommand{\R}{\mathbb{R}} 
\newcommand{\C}{\mathbb{C}} 
\newcommand{\K}{\mathbb{K}} 

\usepackage{amssymb}
\usepackage{array}

\begin{document}

\title{Dynamics and spectra of composition operators on the Schwartz space}
\author{Carmen Fern\'andez, Antonio Galbis, Enrique Jord\'a}
\maketitle

\begin{abstract}
In this paper we study the dynamics of the composition operators defined in the Schwartz space $\mathcal{S}(\R)$ of rapidly decreasing functions. We prove that such an operator is never supercyclic and, for monotonic symbols, it is power bounded only in trivial cases. For a polynomial symbol $\varphi$ of degree greater than one we show that the operator is mean ergodic if and only if it is power bounded and this is the case when $\varphi$ has even degree and lacks fixed points. We also discuss the spectrum of composition operators.
\end{abstract}

\section{Introduction and notation}
We study the dynamics of composition operators defined in the Schwartz space $\mathcal{S}(\R)$ of smooth rapidly decreasing functions. The smooth functions $\varphi:\R\to \R$ for which the composition operator $C_\varphi:\mathcal{S}(\R)\to \mathcal{S}(\R), f\mapsto f\circ \varphi,$ is well defined were characterized  by the second and the third author in \cite{GJ18}, where the compactness and closed range of the operator is analyzed. In this paper we discuss the behavior of the orbits $\{C_{\varphi}^n(f):\ n\in\N\}$. Dynamics of composition operators in Banach spaces of analytic functions on the unit disc have been broadly investigated. There are a lot of results relating the dynamics of $C_\varphi$ to that of $\varphi$ \cite{cowen,shapiro}. In the last years composition operators on spaces of smooth functions on the reals have attracted the attention of several authors. Dynamics on the space of real analytic functions is analyzed by Bonet and Doma\'nski in \cite{bd2,bd3}. More recently, Kennesey, Wengenroth
and Przestacki have investigated composition operators on the space $C^{\infty}(\R)$ of smooth functions on $\R$ (see \cite{kw,adam1,adam2,adam2.5}). The dynamics of composition operators on $C^{\infty}(\R)$ has been studied in \cite{adam3}.
\par\medskip
In \cite{golinski} it is proved that $\mathcal{S}(\R)$ admits continuous linear operators for which every nonzero vector is hypercyclic.   In Section 2 we prove that composition operators on the Schwartz class cannot provide these kind of examples, since they are neither hypercyclic nor supercyclic. We recall that an operator $T$ on a locally convex space (lcs) is said to be hypercyclic if there exist a dense orbit $O(T,x):=\{T^n(x):\ n\in\N\}$. The operator is supercyclic if there exists $x\in X$ such that the projective orbit $\K O(T,x)=\{\lambda T^n(x):\ \lambda\in \K, \ n\in\N\}$ is dense. It follows from the definition that only separable spaces support supercyclic operators. There is a vast literature studying hypercyclicity and supercyclicity in concrete operators defined on Banach or Fr\'echet spaces (see the monographies \cite{Bayart,GE_Peris}).
\par\medskip
An operator $T:X\to X$ is said to be power bounded if $\{T^n: n\in \N\}$ is an equicontinuous set. If $X$ is a Fr\'echet space then $T$ is power bounded if and only if $\{T^n(x):\ n\in\N\}$ is bounded for each $x\in X.$ A closely related concept to power boundedness is that of {\em mean ergodicity}. Given $T\in L(X)$, the Ces\`aro means of $T$ are defined
as $T_{[n]}=\sum_{k=1}^{n}T^k/n$. $T$ is said to be mean ergodic when $T_{[n]}$ converges to an operator $P$, which is always a projection,  in the strong operator topology, i.e. if $(T_{[n]}(x))$ is convergent to $P(x)$ for each $x\in X$. Clearly, if $T$ is mean ergodic then $\lim_{n\to \infty}\frac{T^n(x)}{n} = 0$ for each $x\in E.$ The operator is called {\em uniformly mean ergodic} if this convergence happens uniformly on bounded sets, that is $(T_{[n]})$ is convergent to $P$ in $L_b(X)$. When $X$ is a Banach space this means  that the convergence happens in the operator norm topology. If $X$ is reflexive then each power bounded operator is mean ergodic. The result was proved by Lorch \cite{lorch} for reflexive Banach spaces extending the classical Von Neumann mean ergodic theorem valid for unitary operators defined on a Hilbert space. Albanese, Bonet and Ricker \cite{abr} showed that the result remains true if $X$ is a Fr\'echet space. Moreover, if $X$ is {\em Montel}, i.e. barrelled and such that closed
and bounded subsets are compact, it follows that mean ergodicity and uniform mean ergodicity are equivalent concepts. Power boundedness and mean ergodicity are mainly studied as a theoretical tool for analizing the structure of Banach spaces. For example, Fonf, Lin and Wojstaycyck proved in \cite{FLW} that if $X$ is a Banach space which has Schauder basis and it is not reflexive then there exists an operator which is power bounded but not mean ergodic. In the last years,  power boundedness and mean ergodicity have been studied by several authors from a dynamical point of view,  mainly in spaces of analytic functions \cite{bgjj1,bgjj2,bd1,br}. Roughly speaking, power boundedness and mean ergodicity are related to {\em small orbits}, and hyperciclicity and superciclicity to {\em big orbits}. This size classification is of course relative. If $T$ is power bounded then $T$ cannot be hypercyclic, but  $T$ could be supercyclic. Beltr\'an-Meneu provides in \cite{tesiMJ}  an example suggested by Peris of a
hypercyclic operator which is also mean ergodic.
\par We prove that for an increasing symbol $\varphi$ other than the identity the operator $C_\varphi$ is not power bounded on the Schwartz class. For a decreasing symbol $\varphi$ the operator $C_\varphi$ is power bounded if and only if it is mean ergodic and this only happens when $\varphi\circ\varphi$ is the identity. We completely characterize those polynomials $\varphi$ for which $C_\varphi$ is mean ergodic or power bounded. This is the content of Theorem \ref{theo:polynomial_symbol}.
\par
\medskip
Section 4 deals with the spectrum and pointwise spectrum of composition operators. For injective symbols, the pointwise spectrum is completely characterized in Propositions \ref{prop:spectra_increasing} and \ref{prop:spectra_decreasing}. It turns out that, for injective symbols, the pointwise spectrum is empty or reduces to $\{1\}$ (in case the symbol is increasing) or to $\{-1,1\}$ (if the symbol is decreasing). The behavior is completely different for non-injective symbols as Example \ref{ex:eigenvalues_sqrt} shows, where a composition operator is provided whose pointwise spectrum coincides with the open unit disc. We prove that the spectrum of a mean ergodic composition operator is always contained in the closed unit disc (Corollary \ref{cor:spectra_mean-ergodic}). Concrete examples are given were the spectrum coincides with the open unit disc, the unit circle or ${\mathbb C}\setminus \{0\}.$ The spectrum of composition operators in spaces of analytic functions has been recently considered by many authors
(see for instance \cite{cgp,gs,hlns}).
\par\medskip
Let $X({\mathbb R})$ be a locally convex space of functions defined on ${\mathbb R}.$ Then $\varphi:{\mathbb R}\to {\mathbb R}$ is said to be a symbol for $X({\mathbb R})$ if $C_\varphi, f\mapsto f\circ \varphi,$ maps $X({\mathbb R})$ continuously into itself.
\par\medskip
Recall that $\mathcal{S}(\R)$ consists of those smooth functions $f:\R\to\C$ with the property that
$$ \pi_n(f):=\sup_{x\in \R}\sup_{1\leq j\leq n}(1+|x|^2)^n|f^{(j)}(x)|<\infty\ $$ for each $n\in\N.$ $\mathcal{S}(\R)$ is a Fr\'echet Montel space when endowed with the topology generated by the sequence of seminorms $\left(\pi_n\right)_{n\in \N}.$
\par\medskip
We state below the characterization of the symbols for $\mathcal{S}(\R)$.

\begin{thm}[\cite{GJ18}]
\label{symbols}{\rm
A function $\varphi\in C^{\infty}(\R)$ is a symbol for $\mathcal{S}(\R)$ if and only if the following conditions are satisfied:
\begin{itemize}
\item[(i)] For all $j\in \N_0$ there exist $C,p>0$ such that
$$|\varphi^{(j)}(x)|\leq C(1+|\varphi(x)|^2)^p$$

\noindent for every $x\in\R$.

\item[(ii)] There exists $k>0$ such that $|\varphi(x)|\geq |x|^{1/k}$ for all $|x|\geq k.$

\end{itemize}
}
\end{thm}
\par\medskip
It follows that every symbol $\varphi$ for $\mathcal{S}(\R)$ goes to infinity as $|x|$ goes to infinity.
\par\medskip
From now on $\varphi_n = \varphi\circ\ldots\circ\varphi$ denotes the $n$-th iteration of $\varphi.$ In the case that $\varphi$ is a bijection we also write $\varphi_{-n} = \varphi^{-1}\circ\ldots\circ\varphi^{-1}.$

\section{Supercyclicity of the composition operator}

The composition operators on ${\mathcal S}({\mathbb R})$ are never hyperciclic. In fact, for every symbol $\varphi$ and for every $f\in {\mathcal S}({\mathbb R})$ all the functions in the orbit
$$
O(C_\varphi,f) = \left\{f\circ \varphi_n:\ n\in\N\right\}
$$ have the range contained in the bounded set $f({\mathbb R}).$ Therefore no orbit can be dense. The aim of this section is to check that composition operators are not supercyclic. In what follows we will use the following result due to Bayart and Matheron (\cite[Prop I.26]{Bayart}) relating supercyclicity to hypercyclicity: if $T$ is a supercyclic operator on $X$ and the pointwise spectrum $\sigma_p(T^\ast)$ of $T^\ast$ is non empty then $\sigma_p(T^\ast) = \{\lambda\},\ \lambda\neq 0,$ and there is a closed hyperplane $X_0\subset X$ such that $\lambda^{-1}\cdot T_{|X_0}:X_0\to X_0$ is hyperciclic.

A continuous function $\varphi:\R\to\R$ is said to be proper if $\varphi^{-1}([-M,M])$ is bounded in $\R$ for each $M>0$. Let $C_0(\R)$ be the Banach space of continuous functions on $\R$ which vanish at infinity endowed with the $\|\cdot\|_\infty$-topology, and let $\mathcal{D}(\R)$ be the space of compactly supported smooth functions endowed with its natural locally convex topology. It is well known that the inclusions $\mathcal{D}(\R)\hookrightarrow \mathcal{S}(\R)\hookrightarrow C_0(\R)$ are continuous and have dense range. According to condition (ii) in Theorem \ref{symbols} the symbols for $\mathcal{S}(\R)$ are smooth proper functions, which are precisely the symbols for the composition operators on $\mathcal{D}(\R)$. Continuous proper functions form the set of symbols for composition operators on $C_0(\R)$.

\begin{thm}{\rm
	Let $\varphi:\R \to \R$ be a proper continuous function. Then $$C_{\varphi}:C_0(\R)\to C_0(\R)$$ is not supercyclic. If in addition $X$ is a lcs, $X\hookrightarrow C_0(\R)$ is continuously embedded with dense range and $C_{\varphi}(X)\subseteq X$, then $C_{\varphi}:X\to X$ is not supercyclic.
}
\end{thm}	

\begin{proof}
We only need to prove the first statement and we proceed by contradiction. So, let us assume that $C_\varphi:C_0(\R)\to C_0(\R)$ is supercyclic. If $\varphi$ has a fixed point $x_0$ then $\delta_{x_0}$ is a fixed point of $C_{\varphi}^*.$ Here $\delta_{x_0}$ stands for the evaluation functional at $x_0.$ By \cite[Prop I.26]{Bayart} we get a closed hyperplane $X_0\subset C_0(\R)$ such that $C_{\varphi}(X_0)\subset X_0$ and $C_{\varphi}: X_0\to X_0$ is hypercyclic. Let $x\in \R$ such that $X_0\not\subset Ker\delta_x$ and let $f\in X_0$ be a hypercyclic vector. Then $$\{\delta_x C_{\varphi_n}(f)\ n\in\N\:\} = \{f\left(\varphi_n(x)\right):\ n\in\N\}$$ is a bounded set. This is a contradiction with the fact that $f$ is a hypercyclic vector for $C_{\varphi}|_{X_0}$.
\par
If $\varphi$ is not injective then there are two different real numbers $a, b$ such that, for every $f\in C_0(\R),$ the projective orbit $\K O(C_\varphi,f)$ is contained in the proper closed subspace $\{h\in C_0(\R): h(a) = h(b)\}.$ This is also a contradiction.
\par
Consequently, $\varphi$ is strictly monotone and does not have fixed points. This forces $\varphi$ to be increasing and $\varphi(x)-x$ to have constant sign. The last assertion implies that the sequence $\left(|\varphi_n(x)|\right)_n$ diverges to $\infty$ for each $x\in\R$. From this it follows that the sequence $\left(f\left(\varphi_n(x)\right)\right)_n$ converges to $0$ for each $x\in\R$ and $f\in C_0(\R)$. Let $f\in C_0(\R)$ be a supercyclic vector. We can assume without loss of generality that $\|f\|_\infty=1$. Since $\varphi(\R)=\R$ we have that $C_\varphi$ is an isometry in $C_0(\R)$ and we get

 $$\lim_n \frac{|\langle  \delta_0, f\circ\varphi_n\rangle | }{\|f\circ\varphi_n\|_\infty}=\lim_n f\left(\varphi_n(0)\right) = 0,$$ a contradiction with the angle criterion for supercyclicity \cite[Theorem 9.1]{Bayart}, since $\|\delta_0\|=1$. This completes the proof of the first statement while the second one follows from the definitions.
\end{proof}	

\begin{corollary}{\rm
\begin{itemize}

\item[(1)] If $\varphi$ is a symbol for $\mathcal{S}(\R)$ then $C_\varphi: \mathcal{S}(\R)\to \mathcal{S}(\R)$ is not supercyclic.
\item[(2)] If $\varphi$ is smooth and proper then $C_\varphi: \mathcal{D}(\R)\to \mathcal{D}(\R)$ is not supercyclic.
\end{itemize}
}
\end{corollary}

We remark that in case of non injective symbols we even have that the composition operator is not cyclic, i.e. the image of the composition operator, and therefore the linear span of any orbit is always contained in a proper closed set.

\section{Power bounded and mean ergodic composition operators}

Once we know that composition operators on the Schwartz class cannot have large orbits, we investigate when they are power bounded or mean ergodic. We recall that every power bounded operator on a Fr\'echet-Montel space is uniformly mean ergodic.

\subsection{Some necessary conditions on mean ergodicity}

\begin{lem}\label{lem:Non-meanergodic}{\rm \begin{itemize} \item[(a)] If there exist an unbounded sequence $\left(x_n\right)_n$ such that $\left(\varphi_n(x_n)\right)_n$ is bounded then $C_\varphi$ is not power bounded.
\item[(b)] If $|x_n|^k\geq n$ for some $k\in {\mathbb N}$ and $\left(\varphi_n(x_n)\right)_n$ is bounded then $C_\varphi$ is not mean ergodic.
                  \end{itemize}
 }
\end{lem}
\begin{proof}
  Let us assume that $\left|\varphi_n(x_n)\right| \leq M$ for every $n\in {\mathbb N}$ and take $f\in {\mathcal S}({\mathbb R})$ such that $f = 1$ on $[-M, M].$

  (a) If  $(x_n)_n$ is unbounded, then
$$
\sup_{n\in {\mathbb N}}\left(1 + |x_n|\right)\cdot\left|C_{\varphi_n}f\left(x_n\right)\right| = \infty
$$ and $\left(C_{\varphi_n} f\right)_n$ is unbounded.
\par\medskip
(b) The hypothesis implies that $$\left(1 + |x_n|\right)^k\cdot\frac{\left|f\left(\varphi_n(x_n)\right)\right|}{n} = \left(1 + |x_n|\right)^k\cdot\frac{\left|C_{\varphi_n} f(x_n)\right|}{n}$$ does not converge to $0.$ Hence, $\left(\frac{1}{n}C_{\varphi_n}f\right)_n$ does not converge to $0$ and, consequently, $C_\varphi$ is not mean ergodic.
\end{proof}

\begin{corollary}
\label{mayorquea}{\rm
If $\varphi$ is a symbol for $S(\R)$ with $\varphi(\R)=\R$ and there exists $\delta>0$ such that $|\varphi(x)-x|>\delta$ for all $x\in \R$ then $C_\varphi$ is not mean ergodic.
}
\end{corollary}
\begin{proof}
We consider only the case $\varphi(x)-x>\delta$. Iterating this condition we get $\varphi_n>\varphi_{n-1}+\delta$, hence $\varphi_n(x)>x+n\delta$ for all $x\in\R$. Therefore, $\varphi_n(x)=0$ implies
$x\leq -n\delta$. From $\varphi_n(\R)=\R$ we get a sequence $\left(x_n\right)_n$ satisfying $\varphi_n(x_n)=0$ and $|x_n|^2>n$ for $n$ large enough. The conclusion follows from Lemma \ref{lem:Non-meanergodic} (b).
\end{proof}

For increasing symbols without fixed points, the hypothesis in the previous result can be relaxed.

\begin{corollary}
\label{asymptotic}{\rm Let $\varphi$ be a strictly increasing symbol without fixed points. If there exists $x_0$ such that $\varphi(x)-x>\delta$ for all $x\leq x_0$ then $C_\varphi$ is not mean ergodic. This is the case if $\varphi(x) > x$ for every $x\in {\mathbb R}$ and $\varphi'(x)\leq 1$ for every $x\leq x_0.$
}
\end{corollary}
\begin{proof}
 As $\varphi$ has no fixed points, we must have $\varphi(t)>t,$ or equivalently, $\varphi_{-1}(t)<t, $ for every $t\in \R.$ Then, $x_n:=\varphi_{-n}(x_0)<x_0,$ and therefore, $x_n<x_{n-1}-\delta$ for all $n \in \N,$ hence $x_n < x_0 - n \delta,$ and we can argue as in Corollary \ref{mayorquea}. Let us now assume $\varphi'(x)\leq 1$ for every $x\leq x_0$ and $\varphi(x_0)>x_0.$ Then $\varphi(x)-x$ is decreasing in $(-\infty, x_0]$ and hence, taking $\delta:=(\varphi(x_0)-x_0) > 0,$ we have $\varphi(x)\geq x+\delta$ for each $x\leq x_0.$
\end{proof}

A similar result holds if there exists $x_0$ such that  $x-\varphi(x)>\delta$ for every $x\geq x_0.$ In particular, if $x>\varphi(x)$ for every $x\in \R$ and $\varphi'(x)\leq 1$ for all $x\geq x_0$ then $C_\varphi$ is not mean ergodic.

\par\medskip

 From now on, when $C_\varphi$ is mean ergodic we will put
$$Pf = \lim_{n\to \infty}\frac{1}{n}\sum_{k=1}^n C_{\varphi_n}f.
$$

\begin{proposition}
\label{eventuallydivergent}{\rm
Let $\varphi$ be a symbol with a bounded orbit. Assume that there exists $k>0$ such that $\left(\varphi_n(x)\right)_n$ is unbounded for each $|x|\geq k$. Then $C_\varphi$ is not mean ergodic.
}
\end{proposition}
\begin{proof}
Let us assume that $C_\varphi$ is mean ergodic. Then, as $P$ is a continuous operator and $\lim_n \frac{C_{\varphi_{n+1}}(f)}{n}=0,$ we have that $C_\varphi(Pf)=Pf,$ for each $f \in {\mathcal S}({\mathbb R}).$

By hypothesis the set $A:=\{x\in {\mathbb R}: (\varphi_n(x))_n \mbox{ is bounded}\}$ is non-empty and bounded. Take $b:={\rm sup}\{|x|: \, x\in A\},$ and select a sequence $x_k\in A$ converging to $b$ (or to $-b$). Observe that, for  $n, k \in {\mathbb N},$ $\varphi_n(x_k)\in A\subset[-b,b],$ whereas $(\varphi_n(x))_n$ is unbounded when $|x|>b.$ Then, if $f\in {\mathcal S}({\mathbb R})$, $f\equiv 1$ in $[-b,b]$ (observe that we do not assume $b>0$), we have $$Pf(b)=\lim_{k} Pf(x_k)=1.$$
 On the other hand,  $C_{\varphi}(Pf)=Pf$ implies $Pf\left(\varphi_n(x)\right) = Pf(x)$ for all $x\in {\mathbb R}$ and every $n\in {\mathbb N}.$ Then, $Pf(x)=0$ if the orbit $\left(\varphi_n(x)\right)_n$ is unbounded, in particular $Pf(x)=0$ when $|x|>b,$ a contradiction.
\end{proof}

 Now we concentrate on monotone symbols. For $f\in {\mathcal S}({\mathbb R})$ we will use the convention $f(+\infty) = f(-\infty) = 0.$

 \begin{lem}{\rm
Let $\varphi$ be an increasing symbol. Then, for every $x\in {\mathbb R}$ there exists
$$
\varphi^\ast(x) = \lim_{n\to \infty}\varphi_n(x)\in \overline{{\mathbb R}} = {\mathbb R}\cup \left\{\pm\infty\right\}.
$$ If $C_\varphi$ is mean ergodic then
$$
f\circ \varphi^\ast \in {\mathcal S}({\mathbb R}) \ \ \forall\ f\in {\mathcal S}({\mathbb R}).
$$
}
\end{lem}
\begin{proof}

Given $x\in \R,$ as $\varphi $ is increasing, the orbit $\left(\varphi_n(x)\right)_n$  is an increasing (resp. decreasing) sequence if $\varphi(x)\geq x$ (resp. $\varphi(x)<x$) and therefore convergent in $\overline{{\mathbb R}}.$
Let us now assume that $C_\varphi$ is mean ergodic. Then, for every $f\in {\mathcal S}({\mathbb R})$ we have $Pf\in {\mathcal S}({\mathbb R}),$ but
$$
\left(Pf\right)(x) = \lim_{n\to \infty}\frac{1}{n}\sum_{k=1}^n f\left(\varphi_n(x)\right) = f\left(\varphi^\ast(x)\right),
$$ from where the conclusion follows.
\end{proof}

\begin{proposition}
\label{increasing}
{\rm Let $\varphi$ be an increasing symbol with some fixed point. Then either $\varphi(x) = x$ for every $x\in {\mathbb R}$ or $C_\varphi$ is not mean ergodic.
}
\end{proposition}
\begin{proof}
 Let us denote
 $$
 F = \left\{x\in {\mathbb R}: \ \varphi(x) = x\right\},
 $$ which by hypothesis is not empty. Let us  assume that $F\neq {\mathbb R}.$  We now distinguish three cases.
 \par(a) $F$ is bounded above. Let $x_0$ be the maximum of $F.$ Then, either (a1) $\varphi(x) > x$ for every $x > x_0$ or (a2) $\varphi(x) < x$ for every $x > x_0.$ If (a1) holds then $\varphi^\ast(x) = +\infty$ for every $x > x_0$ and
 $$
 f\left(\varphi^\ast(x)\right) = 0 \ \ \forall x > x_0\ \ \mbox{while}\ f\left(\varphi^\ast(x_0)\right) = f(x_0)
 $$ for every $f\in {\mathcal S}({\mathbb R}).$ In case (a2) holds then
 $$
 f\left(\varphi^\ast(x)\right) = f(x_0)\ \ \forall x \geq x_0, \ \ \forall \ f\in {\mathcal S}({\mathbb R}).
 $$ In both cases, choosing $f\in {\mathcal S}({\mathbb R})$ with $f(x_0)\neq 0$  we see that  $f\circ \varphi^\ast \notin {\mathcal S}({\mathbb R}),$ and $C_\varphi$ is not mean ergodic.
 \par
 (b) $F$ is bounded below. We proceed as in case (a) to get a contradiction.
 \par
 (c) $F$ is neither bounded above nor bounded below. Then there are two fixed points $x_0 < y_0$ such that
 $$
 \left(x_0, y_0\right) \cap F = \emptyset.
 $$ Then $\varphi^\ast(x_0) = x_0$ and $\varphi^\ast(y_0) = y_0.$ On the other hand, either (c1) $\varphi^\ast(x) = x_0$ for every $x \in (x_0, y_0)$ or (c2) $\varphi^\ast(x) = y_0$ for every $x \in (x_0, y_0).$ If (c1) holds then
 $$
  f\left(\varphi^\ast(x)\right) = f(x_0) \ \ \forall x\in [x_0, y_0)\ \ \mbox{while}\ \ f\left(\varphi^\ast(y_0)\right) = f(y_0)
  $$ for every $f\in {\mathcal S}({\mathbb R}).$ Hence, taking $f\in {\mathcal S}(\R)$ with $f(x_0)\neq f(y_0)$, $f\circ \varphi^\ast \notin {\mathcal S}({\mathbb R}),$ and $C_\varphi$ is not mean ergodic. A similar argument works on case (c2).
\end{proof}

\begin{proposition}
\label{decreasing}
{\rm Let $\varphi$ be a decreasing symbol. Then either $\varphi_2(x) = x$ for every $x\in {\mathbb R}$ or $C_\varphi$ is not mean ergodic.
}
\end{proposition}
\begin{proof} If $\varphi_2(x)=x$ for every $x\in \R,$ the orbit $O(C_\varphi, f)$ reduces to $\{f, \, f\circ \varphi \}$ for every $f\in {\mathcal S}(\R)$, hence $C_\varphi$ is power bounded.

Let us assume that $C_\varphi $ is mean ergodic and let $\lambda$ be the unique fixed point of $\varphi.$ We note that $\varphi_2$ is increasing and that it may have many fixed points.  We observe that if $a > \lambda$ is a fixed point of $\varphi_2$ then also $\varphi(a)$ is a fixed point of $\varphi_2$ and $\varphi(a) < \lambda.$ Therefore to show that $\varphi_2(x)=x$ for every $x\in \R$ it is enough to prove that
$$
F: = \left\{a\geq \lambda:\ \varphi_2(a) = a\right\}
$$ coincides with $[\lambda, +\infty). $  We proceed by contradiction .
 \par
(a) If $F$ is bounded above,  let $x_0$ be the maximum of $F.$ Then, either (a1) $\varphi_2(x) > x$ for every $x > x_0$ or (a2) $\varphi_2(x) < x$ for every $x > x_0.$ If (a1) holds then
$$
\lim_{n\to \infty}\varphi_{2n}(x) = +\infty,\ \ \lim_{n\to \infty}\varphi_{2n+1}(x) = -\infty,
$$ from where it follows that
$$
\lim_{n\to \infty}f\left(\varphi_n(x)\right) = 0\ \ \forall x > x_0,$$ and thus $Pf|_{(x_0,+\infty)}\equiv 0$ while
$$
Pf(x_0) = \frac{f(x_0) + f\left(\varphi(x_0)\right)}{2}.
$$  for every $f\in {\mathcal S}(\R).$ If instead (a2) holds then
$$
\lim_{n\to \infty}\varphi_{2n}(x) = x_0,\ \ \lim_{n\to \infty}\varphi_{2n+1}(x) = \varphi(x_0),
$$ from where
$$
Pf(x) = \frac{f(x_0) + f\left(\varphi(x_0)\right)}{2}\ \ \forall x > x_0
$$ for every $f\in {\mathcal S}(\R).$ In any case, taking a function $f\in {\mathcal S}(\R)$ with $f(x_0) + f\left(\varphi(x_0)\right) \neq 0$ we run into a contradiction.
\par
(b) Let us assume now that $F$ is not bounded above but $F\neq [\lambda, +\infty).$ Then there are $\lambda \leq a < b$ two consecutive fixed points of $\varphi_2.$  In the case $\varphi_2(x) > x$ for every $x\in (a,b)$ we obtain
$$
\lim_{n\to \infty}\varphi_{2n}(x) = b,\ \ \lim_{n\to \infty}\varphi_{2n+1}(x) = \varphi(b),
$$ so
$$
Pf(x) = \frac{f(b) + f\left(\varphi(b)\right)}{2} \ \ \ \forall x\in (a,b),
$$ while
$$
Pf(a) = \frac{f(a) + f\left(\varphi(a)\right)}{2}.
$$ This is a contradiction since $\varphi(a)\neq b.$ In the case $\varphi_2(x) > x$ a similar argument gives a contradiction. Consequently $\varphi_2(x) = x$ for every $x\geq \lambda.$
\end{proof}

We illustrate below the fact that, besides the trivial case $\varphi(x)=-x$, there are many more decreasing symbols satisfying $\varphi_2(x)=x$.

\begin{example}{\rm Let $f\in C^\infty({\mathbb R})$ be an even function with $\left|f'(x)\right| \leq a < 1$ for every $x\in {\mathbb R}$ and whose derivatives have polynomial growth. Then the equation
$$
x+y = f(x-y)
$$ defines a decreasing symbol $y = \varphi(x)$ such that $\varphi\left(\varphi(x)\right) = x.$
}
\end{example}
\begin{proof}
We first observe that $g(x) = f(x) - x$ is a strictly decreasing function. Moreover, since $\left|f(x)\right| \leq \left|f(0)\right| + a\left|x\right|,$
$$
\lim_{x\to +\infty}g(x) = -\infty,\ \ \ \lim_{x\to -\infty}g(x) = +\infty.
$$ We claim that for every $x\in {\mathbb R}$ there is a unique solution $y = \varphi(x)$ of the equation $x+y = f(x-y).$ In fact, it is enough to take $y = u+x$ where $u$ is the unique solution of $f(u) - u = 2x.$ An application of the implicit function theorem gives that $\varphi\in C^\infty({\mathbb R}).$
\par\medskip
We now check that $\varphi\left(\varphi(x)\right) = x.$ To this end we denote $\overline{x} = \varphi(x)$ and observe that
$$
\overline{x}+x = f\left(x-\overline{x}\right) = f\left(\overline{x}-x\right),
$$ which implies $\varphi\left(\overline{x}\right) = x.$ From
\begin{equation}\label{eq:implicit}
x+\varphi(x) = f\left(\varphi(x)-x\right)
\end{equation} we get
$$
1 + \varphi'(x) = f'\left(\varphi(x)-x\right)\cdot \left(\varphi'(x) - 1\right)
$$ and
$$
\frac{-2}{1-a} < \varphi'(x) = \frac{-1-f'\left(\varphi(x)-x\right)}{1-f'\left(\varphi(x)-x\right)} < 0.
$$ In particular $\varphi'$ is bounded.
\par\medskip
To finish, we show that $\varphi$ is a symbol for ${\mathcal S}({\mathbb R}).$
\par\medskip
(i) We have
$$
\begin{array}{ll}
\left|x\right| - \left|\varphi(x)\right| & \leq \left|x+\varphi(x)\right| = \left|f\left(\varphi(x)-x\right)\right| \\ &\\  & \leq \left|f(0)\right| + a\left|\varphi(x)-x\right| \\ & \\ & \leq \left|f(0)\right| + a\left|\varphi(x)\right| + a\left|x\right|,
\end{array}
$$ from where it follows that there is $\varepsilon > 0$ such that
$$
\left|\varphi(x)\right| \geq \varepsilon \left|x\right| \ \ \forall \left|x\right| \geq \varepsilon^{-1}.
$$
\par\medskip
(ii) Fix $n\geq 2$ and apply Fa\`a di Bruno formula to $(\ref{eq:implicit}).$ Then we obtain
$$
\varphi^{(n)}(x) = \sum \frac{n!}{k_1!\ldots
k_n!}f^{(k)}\left(\varphi(x)-x\right)\left(\frac{\varphi'(x)-1}{1!}\right)^{k_1}\ldots
\left(\frac{\varphi^{(n)}(x)}{n!}\right)^{k_n}$$
\noindent
where the sum is extended over all $(k_1,\ldots,k_n)\in{\mathbb N}_0^n$ such that
$k_1 + 2k_2 + \ldots + nk_n = n$ and $k:=k_1 + \ldots + k_n.$ Then
$$
\varphi^{(n)}(x) = f'\left(\varphi(x)-x\right)\cdot \varphi^{(n)}(x) + g_n(x),
$$ where $g_n(x)$ is a sum involving products of terms of the form $f^{(k)}\left(\varphi(x)-x\right)$ and $\varphi^{(m)}(x)$ and $m < n.$ This permits to use an induction argument to conclude that, for every $n\in {\mathbb N}$ there are $C_n > 0$ and $p_n\in {\mathbb N}$ such that
$$
\left|\varphi^{(n)}(x)\right| \leq C_n \left(1+\left|\varphi'(x)\right|\right)^{p_n}.
$$ Since $\varphi'$ is bounded we conclude.
\end{proof}

We see below that for monotonic symbols power boundedness is equivalent to periodicity of the composition operator.

\begin{thm}\label{monotonic}{\rm
Let $\varphi$ be a monotonic symbol
\begin{itemize}
\item[(a)]  If $\varphi$ is increasing then $C_\varphi$ is power bounded if and only if $\varphi(x)=x$ for each $x\in\R$.
\item[(b)] If $\varphi$ is decreasing then $C_\varphi$ is power bounded if and only if $C_\varphi$ is mean ergodic if and only if $\varphi_2(x)=x$ for each $x\in\R$.
\end{itemize}
}
\end{thm}
\begin{proof}
We get (b) as a direct consequence of Proposition \ref{decreasing}.

To prove (a) we observe that by Proposition \ref{increasing}, if $\varphi$ is an increasing symbol which is not the identity and has at least one fixed point, then $C_\varphi$ is not mean ergodic, and consequently $C_\varphi$ is not power bounded. If $\varphi$ does not have a fixed point we can assume, without loss of generality, $\varphi(x) > x$ for every $x\in {\mathbb R}.$ Since $\varphi({\mathbb R}) = {\mathbb R},$ for every $n\in {\mathbb N}$ there exists $x_n\in {\mathbb R}$ such that $\varphi_n(x_n) = 0.$ From $\varphi_{n+1}(x_{n}) > \varphi(0) > 0$ and $\varphi_{n+1}(x_{n+1}) = 0$ we get $x_{n+1} < x_n.$ Moreover, $\left(x_n\right)_n$ is unbounded. In fact, for every $a\in {\mathbb R}$ we have $\lim_{n\to\infty}\varphi_n(a) = +\infty,$ hence $x_n\leq a$ for $n$ large enough. We conclude that the operator $C_\varphi$ is not power bounded as a consequence of Lemma \ref{lem:Non-meanergodic}.
\end{proof}
\begin{rmk}{\rm
As a consequence of Corollary \ref{asymptotic}, the unique case of monotonic symbols for which we cannot assure that the corresponding composition operators are mean ergodic or not is that of increasing symbols without fixed points that approach asymptotically to $y=x$ as $x$ goes to $-\infty$ (in the case $\varphi(x)>x$) or to $+\infty$ (if $\varphi(x)<x$). We do not know if $C_{\varphi}$ is mean ergodic for $\varphi(x)=x+e^{-x^2}$. On the other hand, for this $\varphi,$ $C_\varphi$ is mean ergodic on $C_0({\mathbb R}).$ In fact, it is power bounded on $C_0({\mathbb R})$ and it is easy to see that for every continuous and compactly supported $f$ one has
$$
\lim_{N\to \infty}\frac{1}{N}\sum_{n=1}^N C_{\varphi_n}f = 0.
$$
}
\end{rmk}

\subsection{Characterization of power boundedness}

Our next result shows that $C_\varphi$ is power bounded if the sequence of symbols $\{\varphi_n:\ n\in {\mathbb N}\}$ satisfies uniformly the conditions in Theorem \ref{symbols}.

\begin{proposition}\label{prop:powerbounded}{\rm For a symbol $\varphi$ the composition operator $C_\varphi$ is power bounded if and only if the following statements hold

\begin{itemize}
\item[(i)] For all $j\in {\mathbb N}_0$ there exist $C,p>0$ such that
$$\left|(\varphi_n)^{(j)}(x)\right|\leq C(1+\varphi_n(x)^2)^p$$ for every $x\in{\mathbb R}$ and  every $n\in {\mathbb N}.$
\item[(ii)] There exists $k>0$ such that $|\varphi_n(x)|\geq |x|^{1/k}$ for all $|x|\geq k$ and  every $n\in {\mathbb N}.$
\end{itemize}
}
\end{proposition}
\begin{proof}
Proceeding as in the proof of \cite[Theorem 2.3]{GJ18} we obtain the sufficiency of conditions (i) and (ii).

Assume now that $C_\varphi$ is power bounded. If (ii) fails, for each $k$ we find $\ell_k\in {\mathbb N}$ and $y_k\in {\mathbb R}$ with $|y_k|>k$ such that
$$
|\varphi_{\ell_k}(y_k)|^k<|y_k|.
$$ As every $\varphi_n$ is a symbol, we may proceed inductively using Theorem \ref{symbols}(ii) and the fact above to find strictly increasing sequences $(k_j)_j$ and $(n_j)_j$ in ${\mathbb N}$ and a sequence $(x_j)_j$ in ${\mathbb R}$ such that $|x_{j}|>k_{j}$ for which $$|\varphi_{n_j}(x_j)|^{k_j}<|x_j|.$$
By Lemma \ref{lem:Non-meanergodic} (a), the sequence $(\varphi_{n_j}(x_j))_j$ is unbounded, hence, passing to a subsequence, we may assume that
$$|\varphi_{n_{j+1}}(x_{j+1})|>1+|\varphi_{n_j}(x_j)|.
$$ Taking  $\rho \in {\mathcal D}\left(-1/2, 1/2\right)$ with $\rho(0)=1,$ the function
$$
f(x):=\sum_{j=1}^\infty \frac{1}{(1+|\varphi_{n_j}(x_j)|^2)^{k_j}}\rho(x-\varphi_{n_j}(x_j)),$$ belongs to ${\mathcal S}({\mathbb R}),$ and $$(1+|x_j|)^3|C^{n_j}_\varphi f(x_j)|\gtrsim 1+|x_j|,$$ a contradiction.

To prove that (i) holds, we proceed again by contradiction. Therefore, we assume that there is $n\in {\mathbb N}$ such that for every $j$ we find $\ell_j$ and $x_j$ with $$|(\varphi_{\ell_j})^{(n)}(x_j)|>j(1+|\varphi_{\ell_j}(x_j)|^2)^j.$$ We can assume that either $n=1$ or there are $p$ and $C$ with
\begin{equation}\label{eq:power}
|(\varphi_\ell)^{(i)}(x)|\leq C(1+|\varphi_\ell(x)|^2)^{p},
\end{equation} for all $x\in {\mathbb R},$ all $\ell \in {\mathbb N}$ and $1\leq i <n.$ Now, if the sequence $(\varphi_{\ell_j} (x_j))_j$ is unbounded we can assume, passing to a subsequence if necessary, that
$$
|\varphi_{\ell_{j+1}}(x_{j+1})|>1+|\varphi_{\ell_j}(x_j)|.
$$ We take $\rho \in {\mathcal D}\left(-1/2, 1/2\right)$ with $\rho'(0)=1$ and $\rho^{(j)}(0)=0$ for $2\leq j \leq n,$ and define
$$
f(x):=\sum_{j=1}^\infty \frac{1}{(1+|\varphi_{\ell_j}(x_j)|^2)^{j}}\rho(x-\varphi_{\ell_j}(x_j)).$$
Then, $f\in {\mathcal S}({\mathbb R}),$ and by Fa\`a di Bruno formula,
$$\left|(C_{\varphi_{\ell_j}}(f))^{(n)}(x_j)\right|=\left|\frac{\rho'(0)(\varphi_{\ell_j})^{(n)}(x_j)|}{(1+(\varphi_{\ell_j}(x_j))^2)^j}\right|>j,$$ a contradiction. Hence, the sequence $(\varphi_{\ell_j} (x_j))_j$  is bounded and, passing to a subsequence, we may assume that it converges to some $y_0\in {\mathbb R}.$ If we take, $f\in {\mathcal S}({\mathbb R}),$ with $f'(y_0)=1,$ then for $j$ big enough we have $f'(\varphi_{\ell_j}(x_j))>1/2.$ By Fa\`a de Bruno formula and (\ref{eq:power}),
$$
\left(C_{\varphi_{\ell_j}}f\right)^{(n)}(x_j)= (\varphi_{\ell_j})^{(n)}(x_j)f'(\varphi_{\ell_j}(x_j))+A_j,$$
where $(A_j)_j$ is a bounded sequence, hence $\left((C_{\varphi_{\ell_j}}f)^{(n)}(x_j)\right)_j$ is unbounded, a contradiction.

\end{proof}

\begin{example}\label{ex:powerbounded_sqrt}{\rm Let $\varphi(x) = \sqrt{x^2 + 1}.$ Then $C_\varphi$ is power bounded.
}
\end{example}
\begin{proof}
 The iterates of the symbol are given by
$\varphi_n(x) = \sqrt{x^2 + n},$ hence $\varphi_n(x)>|x|$ for every $x\in {\mathbb R}$ and condition (ii) is fulfilled. To check condition (i) we see that for every $j\in {\mathbb N}$ there is $C>0$ such that for each $n\in {\mathbb N}$ $$\left|(\varphi_n)^{(j)}(x)\right|\leq C \varphi_n(x).$$

In fact, it is easy to see that $|(\varphi_n)^{(k)}(x)|\leq 1$ for $k=1,\, 2,$ and every $n.$ For $k\geq 3,$ by Fa\`a de Bruno formula, for every $x\in {\mathbb R}$,
$$\left|(\varphi_n)^{(k)}(x)\right|=\sum \frac{k!}{j_1! j_2!}C_{j_1+j_2}(x^2+n)^{-(j_1+j_2)+\frac{1}{2}}|x|^{j_1}\leq D_k \varphi_n(x),$$ where the sum is extended over all $(j_1,j_2)\in {\mathbb N}_0^2$ such that $j_1 + 2j_2 = k.$
\end{proof}

Our next aim is to characterize those polynomials $\varphi$ with the property that $C_\varphi$ is mean ergodic.

\begin{rmk}{\rm Let $\varphi(x) = ax + b$ be given. Then $C_\varphi$ is mean ergodic if and only if $\varphi(x) = x$ or $\varphi(x) = -x + b.$ In fact, let us assume that $a > 0$ and $C_\varphi$ is mean ergodic. From Proposition \ref{increasing} we have that $a = 1.$ Now, from $\varphi_n(-nb) = 0$ and Lemma \ref{lem:Non-meanergodic}(b) we conclude $b = 0.$ In the case that $a < 0$ we apply Proposition \ref{decreasing}.
 }
\end{rmk}

\par\medskip
The next lemma provides information on the behavior of the iterates of polynomials with even degree and without fixed points.

\begin{lem}\label{lem:iterates-symbol}{\rm Let $\varphi$ be a polynomial of even degree without fixed points. Then there is $N\in {\mathbb N}$ such that
$\psi = \varphi_{N}$ has neither zeros nor fixed points. Moreover, for every $K > 0$ there is $m_0\in {\mathbb N}$ such that
$$
\left|\psi_{m+1}(t)\right| \geq K \left(\psi_{m}(t)\right)^2\ \ \forall m\geq m_0,\ \ \ \forall t\in {\mathbb R}.
$$
}
\end{lem}
\begin{proof}
 We assume that $\varphi(t) > t$ for every $t\in {\mathbb R}$ and put $\Phi(t) = \varphi(t)-t.$ Since
 $$
 \lim_{t\to -\infty}\Phi(t) = \lim_{t\to +\infty}\Phi(t) = +\infty
 $$ then there is $a > 0$ such that $\Phi(t) \geq a$ for every $t\in {\mathbb R}.$ An induction argument then gives
 $$
 \varphi_{n}(t) \geq t + na
 $$ for every $n\in {\mathbb N}$ and $t\in {\mathbb R}.$ Consequently
 $$
 \varphi_{n}(t) = \varphi_{n-1}\left(\varphi(t)\right) \geq \varphi(t) + (n-1)a \geq \varphi(t_0) + (n-1)a
 $$ where $\varphi(t_0)$ is the minimum of $\varphi.$ Hence, for $N$ big enough the iterate $\psi = \varphi_{N}$ is a polynomial of even degree greater than 4 without zeros nor fixed points.
 \par\medskip
 Let us now fix $K > 0$ and observe that there is $b > 0$ such that $\psi(s)\geq K s^2$ whenever $s\geq b.$ We choose $m_0$ so that
 $$
 \psi_{m}(t)\geq b
 $$ for all $t\in {\mathbb R}$ and for all $m\geq m_0.$ Then
 $$
\psi_{m+1}(t) = \psi\left(\psi_{m}(t)\right) \geq K \left(\psi_{m}(t)\right)^2.
 $$ The case $\varphi(t) < t$ for every $t\in {\mathbb R}$ is treated in a similar way.
\end{proof}

\begin{thm}\label{theo:polynomial_symbol}{\rm Let $\varphi$ be a polynomial with degree greater than or equal to two. Then, the following are equivalent:
\begin{itemize}
\item[(1)] $C_\varphi$ is power bounded.
\item[(2)] $C_\varphi$ is mean ergodic.
\item[(3)] The degree of $\varphi$ is even and it has no fixed points.
\end{itemize}
}
\end{thm}
\begin{proof}
$(2)\Rightarrow (3).$ We assume that $C_\varphi$ is mean ergodic but $\varphi$ has some fixed point. Then there is at least one point with bounded orbit. Moreover, since $\varphi$ is a polynomial of degree greater than or equal to two we may find $K>0$ such that $|t|>K$ implies $|\varphi(t)| > 2|t|,$ hence $(|\varphi_n(t)|)_n$ goes to infinity for $|t|>K.$ According to Proposition \ref{eventuallydivergent} $C_\varphi$ is not mean ergodic, which is a contradiction. As every polynomial with odd degree greater than one has at least a fixed point, the proof is complete.
\par\medskip
$(3)\Rightarrow (1).$ We first observe that $C_\varphi$ is power bounded if and only if there is $N$ such that $C^N_\varphi$ is power bounded. So, after replacing $\varphi$ by $\varphi_{N}$ for appropriate $N$ we can assume, without loss of generality, that $\varphi$ has neither zeros nor fixed points and enjoy the additional property that $\left|\varphi(t)\right| \geq 1$ and for every $K > 0$ there is $m_0\in {\mathbb N}$ such that
$$
\left|\varphi_{m+1}(t)\right| \geq K \left(\varphi_{m}(t)\right)^2\ \ \forall m\geq m_0,\ \ \ \forall t\in {\mathbb R}
$$ (see Lemma \ref{lem:iterates-symbol}). Condition (ii) in Proposition \ref{prop:powerbounded} is automatically satisfied.

As $\varphi$ is a polynomial which does not vanish, there is $C\geq 1$ such that $|\varphi^{(j)}(t)|\leq C |\varphi(t)|$ for every $t\in {\mathbb R}$ and $j\in {\mathbb N}.$ Let us write
$$
C_n = C\sum \frac{1}{k_1!\dots k_n!}$$ where the sum is extended to all multi-indices such that
$$
 k_1 + 2k_2 + \ldots + nk_n = n.$$ We may find an increasing sequence $(m_n)$ of natural numbers with the property that
$$
\left|\varphi_{m+1}(t)\right| \geq C_n \left(\varphi_{m}(t)\right)^2\ \ \forall m\geq m_n,\ \ \ \forall t\in {\mathbb R}.
$$
\par\medskip
We {\bf claim} that there is $B_n > C_n$ such that for $m\geq m_n,$ and $n \in {\mathbb N},$
\begin{equation}\label{eq:induction}
 \left|\frac{(\varphi_{m})^{(n)}(t)}{n!}\right|\leq B_n^n \left|\varphi_{m}(t)\right|^{2n}\end{equation} for every $t\in {\mathbb R}.$ First, for $n = 1$ we take $B_1$ such that the previous inequality is satisfied for $m = m_1.$ Now, assuming that the inequality holds for $n = 1$ and some $m \geq m_1$ we obtain
$$
\begin{array}{*2{>{\displaystyle}l}} \left|\frac{(\varphi_{m+1})'(t)}{(\varphi_{m+1}(t))^2}\right| & = \frac{|\varphi'(\varphi_{m}(t))\cdot (\varphi_{m})'(t)|}{|\varphi_{m+1}(t)|^2} \leq \frac{C |\varphi_{m+1}(t)|B_1|\varphi_{m}(t)|^2}{|\varphi_{m+1}(t)|^2}\\ & \\ & = C B_1\frac{|\varphi_{m}(t)|^2}{|\varphi_{m+1}(t)|}\leq B_1.\end{array}
$$ Consequently (\ref{eq:induction}) holds for $n =1$ and $m\geq m_1.$ We now take $p\geq 2$ and assume that (\ref{eq:induction}) holds for $n = 1,\ldots, p-1.$ We take
$$
B_p > \max\{C_p, B_1, \dots, B_{p-1}\}$$ such that
$$
\left|(\varphi_{m_p})^{(j)}(t)\right|\leq  B_p^j \left(\varphi_{m_p}(t)\right)^{2j}
$$ for every $j\in {\mathbb N}$ and $t\in {\mathbb R}.$ This selection can be done since $\varphi_{m_p}$ is a polynomial. In particular, condition $(\ref{eq:induction})$ holds for $n = p$ and $m = m_p.$ Now, we assume that condition $(\ref{eq:induction})$ holds for $n = p$ and some $m\geq m_p.$ According to our induction hypothesis, the condition is also satisfied for this $m$ and every derivative or order less than $p.$ Recall that
$$
\begin{array}{*2{>{\displaystyle}l}}
\frac{1}{p!}\left|(\varphi_{m+1})^{(p)}(t)\right| & \leq \sum \frac{1}{k_1!\dots k_m!}|\varphi^{(k)}(\varphi_{m}(t))|\prod_j \left(\frac{|(\varphi_{m})^{(j)}(t)|}{j!}\right)^{k_j}\\ & \\ & \leq \sum \frac{1}{k_1!\dots k_m!}C|\varphi(\varphi_{m}(t))|\prod_j \left(B_j^{jk_j}\cdot \left(\varphi_{m}(t)\right)^{2jk_j}\right),
\end{array}
$$ where the sum is extended to all multi-indices such that $$k_1 + 2k_2 + \ldots + pk_p = p.$$ Then
$$
\begin{array}{*2{>{\displaystyle}l}}
\left|\frac{(\varphi_{m+1})^{(p)}(t)}{p!(\varphi_{m+1}(t))^{2p}}\right|& \leq C_p |\varphi_{m+1}(t)|B_p^p\frac{|\varphi_{m}(t)|^{2p}}{|\varphi_{m+1}(t)|^{2p}}\\ & \\ & = C_p B_p^p\frac{|\varphi_{m}(t)|^{2p}}{|\varphi_{m+1}(t)|^{2p-1}}\leq B_p^p.
\end{array}
$$ The last inequality follows from the fact that $\left|\varphi_{m+1}(t)\right| \geq C_p\left(\varphi_{m}(t)\right)^2$ and $\left|\varphi_{m}(t)\right|\geq 1.$ The claim is proved. Now it easily follows condition (i) in Proposition \ref{prop:powerbounded}. Hence $C_\varphi$ is power bounded.
\end{proof}

\begin{corollary}{\rm Let $\varphi$ be a polynomial with degree greater than one. Then, the following are equivalent:
\begin{itemize}
\item[(1)] $C_\varphi$ is power bounded.
\item[(2)] $C_\varphi$ is mean ergodic.
\item[(3)] $\left(C_{\varphi_n}f\right)_n$ converges to zero as $n$ goes to infinity for every $f\in {\mathcal S}({\mathbb R}).$
\end{itemize}
}
\end{corollary}
\begin{proof}
 It suffices to show $(1)\Rightarrow (3).$ It follows from the proof of Theorem \ref{theo:polynomial_symbol} that there is $N\in {\mathbb N}$ such that $\left(\varphi_{Nk}(t)\right)_k$ goes to infinity for every $t\in {\mathbb R}.$ Since every orbit of $C_\varphi$ is relatively compact we easily conclude.
\end{proof}

Next result can be proved with the same ideas of Theorem \ref{theo:polynomial_symbol} but it is technically more difficult.

\begin{thm}\label{cor:exponencial-pol}{\rm
Let $\varphi$ be a polynomial of even degree without fixed points. If $\varphi$ is bounded below and $\psi:=e^{\varphi}$ then $C_{\psi}$ is power bounded.
}
\end{thm}

\section{Spectra of composition operators}

In what follows ${\mathbb D}$ will always denote the open unit disc in ${\mathbb C}.$ We recall that $\sigma_p(C_\varphi)$ is the set of all eigenvalues of $C_\varphi$ and $\sigma(C_\varphi)$ is the spectrum of $C_\varphi,$ that is the set of all $\lambda\in {\mathbb C}$ such that $C_\varphi - \lambda \mbox{Id}:{\mathcal S}({\mathbb R})\to {\mathcal S}({\mathbb R})$ does not admit a continuous linear inverse.

\begin{proposition}\label{prop:eigenvalues_disc}{\rm For every symbol $\varphi$ we have $\sigma_p(C_\varphi) \subset \overline{{\mathbb D}}.$ If all the orbits of $\varphi$ are unbounded then $\sigma_p(C_\varphi) \subset {\mathbb D}.$
 }
\end{proposition}
\begin{proof} Take $f\in {\mathcal S}({\mathbb R})\setminus\left\{0\right\}$ and $\lambda \in {\mathbb C}$ such that $C_\varphi(f) = \lambda f$ and fix $t_0\in {\mathbb R}$ with the property that $f(t_0)\neq 0.$ From the identity
$$
f\left(\varphi_n(t_0)\right) = \lambda^n f(t_0)
$$ we deduce that $\left(\lambda^n\right)_{n\in {\mathbb N}}$ is a bounded sequence, from where the conclusion follows. In the case that $\left(\varphi_n(t_0)\right)_{n\in {\mathbb N}}$ is unbounded we can take a subsequence $\left(\varphi_{n_k}(t_0)\right)_k$ diverging to infinity, hence
$$
\lim_{k\to \infty}f\left(\varphi_{n_k}(t_0)\right) = 0,
$$ and $|\lambda| < 1.$
\end{proof}

In the case that the symbol is injective, more can be said.

\begin{proposition}\label{prop:spectra_increasing}{\rm Let $\varphi$ be an strictly increasing symbol for ${\mathcal S}({\mathbb R}).$ Then
\begin{itemize}
 \item[(a)] $\sigma_p\left(C_\varphi\right)\subset \left\{1\right\}.$
\item[(b)] $\sigma_p\left(C_\varphi\right) = \left\{1\right\}$ if and only if the set $\left\{t: \varphi(t) = t\right\}$ has interior points.
\end{itemize}
}
\end{proposition}
\begin{proof}
(a) Let $\lambda\in {\mathbb C}$ and $f\in {\mathcal S}({\mathbb R})\setminus \left\{0\right\}$ be given such that $C_\varphi(f) = \lambda f.$ Since $\varphi:{\mathbb R}\to {\mathbb R}$ is a bijection then $\lambda \neq 0.$ Denote $\psi = \varphi^{-1},$ fix $t\in {\mathbb R}$ with $f(t)\neq 0$ and observe that
$$
f\left(\varphi_n(t)\right) = \lambda^n f(t)\ \ \mbox{and}\ \ f\left(\psi_n(t)\right) = \lambda^{-n} f(t)
$$ for every $t\in {\mathbb R}.$ Since the orbits $\left(\varphi_n(t)\right)_n$ and $\left(\psi_n(t)\right)_n$ are monotone and convergent in ${\mathbb R}\cup \left\{\pm \infty\right\}$ and $f$ is continuous and vanishes at infinity we conclude that both sequences $\left(\lambda^n\right)_n$ and $\left(\lambda^{-n}\right)_n$ are convergent. Hence $\lambda = 1.$
\par\medskip\noindent
(b) Let us first assume that $f\in {\mathcal S}({\mathbb R})\setminus \left\{0\right\}$ and $C_\varphi(f) = f.$ Take $I = (a,b)$ with the property that $f^\prime(t)\neq 0$ for every $t\in I.$ We now show that $\varphi(t) = t$ for every $t\in I.$ Otherwise, there are $t\in I$ and $\varepsilon > 0$ such that $\varphi(s) - s$ has constant sign in the interval $J = (t-\varepsilon, t+\varepsilon)\subset I.$ All the sequences $\left(\varphi_n(s)\right)_n, s\in J,$ are monotone and convergent to the same point in ${\mathbb R}\cup \left\{\pm \infty\right\}.$ From the identity $f\left(\varphi_n(s)\right) = f(s)$ for every $n\in {\mathbb N}$ and $s\in J$ we deduce, after taking limits as $n$ goes to infinity, that $f$ is constant in $J,$ which is a contradiction.
\par\medskip
Let us now assume that there is an open interval $I = (a,b)$ such that $\varphi(t) = t$ whenever $t\in I.$ Then every $f\in {\mathcal D}(a,b)$ satisfies $C_\varphi(f) = f.$ This follows from the fact that $\varphi\left({\mathbb R}\setminus I\right) = {\mathbb R}\setminus I.$
\end{proof}

\begin{proposition}\label{prop:spectra_decreasing}{\rm Let $\varphi$ be an strictly decreasing symbol for ${\mathcal S}({\mathbb R}).$ Then, the following are equivalent
\begin{itemize}
 \item[(a)] $\sigma_p\left(C_\varphi\right)\neq \emptyset.$
\item[(b)] The set $\left\{t: \varphi_2(t) = t\right\}$ has interior points.
\item[(c)] $\sigma_p\left(C_\varphi\right) = \left\{-1,1\right\}.$
\end{itemize}
}
\end{proposition}
\begin{proof}
$(a)\Rightarrow (b)$ $\varphi_2$ is an strictly increasing symbol for ${\mathcal S}({\mathbb R})$ and $\sigma_p\left(C_{\varphi_2}\right)\neq \emptyset.$ Now we apply Proposition \ref{prop:spectra_increasing}.
\par\medskip\noindent
$(b)\Rightarrow (c)$ From Proposition \ref{prop:spectra_increasing} we have $\sigma_p\left(C_{\varphi_2}\right)\subset \{1\},$ hence $\sigma_p\left(C_{\varphi}\right)\subset \{-1, 1\}.$  Let $(a,b)$ be an open interval such that $\varphi_2(t) = t$ for every $t\in (a,b).$ Fix $t_0\in (a,b)$ with $\varphi(t_0)\neq t_0$ (note that $\varphi$ has exactly one fixed point) and choose $\varepsilon > 0$ with the property that $I = \left(t_0-\varepsilon, t_0+\varepsilon\right) \subset (a,b)$ and the distance between the open intervals $I$ and $J = \varphi(I)$ is strictly positive. We observe that $\varphi(J) = I.$ Now, for $\lambda \in \left\{-1,1\right\},$ choose an arbitrary $g\in {\mathcal D}(J), g\neq 0,$ and define $h\in {\mathcal D}(I)$ by
$$
\lambda h\left(t\right) = g\left(\varphi(t)\right),\ t\in I.
$$ Finally we consider the function $f\in {\mathcal S}({\mathbb R})$ defined by $f = g$ on $J,$ $f = h$ on $I$ and $f = 0$ elsewhere. From the definition of $h$ and the fact that $\lambda^{-1} = \lambda$ we obtain $f\left(\varphi(t)\right) = \lambda f(t)$ for every $t\in {\mathbb R}.$
\end{proof}

\par\medskip
In the case that the symbol is not injective the situation can be completely different, as the following example shows.

\begin{example}\label{ex:eigenvalues_sqrt}{\rm Let $\varphi(x) = \sqrt{x^2 + 1}.$ Then
$$
\sigma_p(C_\varphi) = {\mathbb D}.
$$
}
\end{example}
\begin{proof} The iterates of the symbol are given by
$$
\varphi_n(x) = \sqrt{x^2 + n}.
$$ Since all the orbits of $\varphi$ are unbounded we can apply Proposition \ref{prop:eigenvalues_disc} to conclude $\sigma_p(C_\varphi)\subset {\mathbb D}.$ Hence it suffices to show that ${\mathbb D} \subset \sigma_p(C_\varphi).$ For convenience we denote $\varphi_0(x) = x.$ Take $I = (x_0, y_0) = \left(\frac{1}{4}, \frac{1}{2}\right)$ and observe that $$\varphi_n(I) = \left(x_n, y_n\right)\ \ \mbox{where}\ \ 0 < x_0 < y_0 < x_1 < y_1 < x_2 < y_2 < \ldots
$$ We now fix $\lambda \in{\mathbb C}$ with $|\lambda| <1$ and take a test function $\psi\in {\mathcal D}(I).$ For $x\geq 0$ we define
$$
f(x) = \lambda^n \left(\psi\circ\left(\varphi_n\right)^{-1}\right)(x)\ \ \mbox{if}\ \ x\in\varphi_n(I)\ \ (n = 0,1,2,\ldots)
$$ and
$$
f(x) = 0 \ \ \mbox{in the case that}\ \ x\geq 0, x\notin \bigcup_{n=0}^\infty\varphi_n(I).
$$ We observe that $f(x) = 0$ for every $x$ in a neighborhood of the points $\left\{x_n,y_n\right\}_{n=0}^\infty.$ Finally we extend $f$ to the negative real numbers by $f(-x) = f(x).$ Then $f$ is an smooth function. Our aim is to check that $$ \mbox{(a)}\ f\in {\mathcal S}({\mathbb R})\ \ \mbox{and}\ \ \mbox{(b)}\ C_\varphi(f) = \lambda f.$$
\par
(a) $f\in {\mathcal S}({\mathbb R}).$ Since $(1+x^2)^k \leq \left(n+\frac{5}{4}\right)^k$ for every $x\in \varphi_n(I)$ and
$$
\lim_{n\to \infty} \lambda^n \left(n+\frac{5}{4}\right)^k = 0
$$ we conclude that
$$
\sup_{x\in{\mathbb R}}\left(1+x^2\right)^k\left|f(x)\right| < \infty.
$$In order to control the derivatives of $f,$ we observe that
$$
\left(\varphi_n\right)^{-1}(x) = g(x^2-n),\ \ x\in \varphi_n(I),
$$ where $g(t) = \sqrt{t}.$ According to Fa\`a di Bruno formula
$$
\left(\left(\varphi_n\right)^{-1}\right)^{(j)}(x) = \sum_{j_1 + 2j_2 = j}\frac{j!}{j_1! j_2!}g^{(j_1 + j_2)}(x^2-n)\left(2x\right)^{j_1}.
$$ For every $k\in {\mathbb R}$ there is $C_k > 0$ such that
$$
\left|g^{(k)}(x^2-n)\right| \leq C_k \ \ \forall x\in \varphi_n(I),\ n\in {\mathbb N}.
$$This is so because $x\in \varphi_n(I)$ implies $x^2 - n\in \left(\frac{1}{16}, \frac{1}{4}\right).$ Consequently, there are constants $B_j$ so that
$$
\left|\left(\left(\varphi_n\right)^{-1}\right)^{(j)}(x)\right| \leq B_j \left|x\right|^j\ \ \forall x\in \bigcup_{n=0}^\infty\varphi_n(I).
$$ Now, a new application of Fa\`a di Bruno formula permits us to conclude that
$$
\sup_{x\in {\mathbb R}}\left(1+x^2\right)^k\left|f^{(n)}(x)\right| < \infty \ \ \forall k,n\in {\mathbb N}.
$$
\par
(b) To check $C_\varphi(f) = \lambda f$ it suffices to show that $f\left(\varphi(x)\right) = \lambda f(x)$ for every $x\geq 0.$ We distinguish two cases. If $x\in \varphi_n(I)$ then $\varphi(x)\in \varphi^{n+1}(I)$ and
$$
f\left(\varphi(x)\right) = \lambda^{n+1}\left(\psi\circ\left(\varphi^{n+1}\right)^{-1}\right)\left(\varphi(x)\right) = \lambda^{n+1}\left(\psi\circ\left(\varphi_n\right)^{-1}\right)(x) = \lambda f(x).
$$ On the other hand
$$
x\notin \bigcup_{n=0}^\infty\varphi_n(I) \Rightarrow \varphi(x) \notin \bigcup_{n=0}^\infty\varphi_n(I),
$$ from where it follows $f\left(\varphi(x)\right) = f(x) = 0.$ Here we are using that $\varphi$ is injective in $(0,\infty)$ and $\varphi(x)\notin I$ for every $x\in {\mathbb R}.$
\end{proof}

\begin{proposition}\label{prop:surjective}{\rm Let $E$ be a Fr\'echet space and $T:E\to E$ a continuous and linear operator. We assume that for every $x\in E$ and for every continuous seminorm $p$ on $E$ there is $\ell\in {\mathbb N}$ such that $p\left(T^n(x)\right) = O(n^\ell)$ as $n\to \infty.$ Then $T-\lambda I$ is surjective for every $\lambda\in {\mathbb C}$ with $|\lambda| > 1.$
 }
\end{proposition}
\begin{proof}
For every $x\in E,$ we consider
$$
y:= - \sum_{n=0}^\infty\frac{1}{\lambda^{n+1}}T^n(x).
$$ The hypothesis on $T$ implies the absolute convergence of the previous series and it is easy to check that $Ty - \lambda y  = x.$
\end{proof}

The next result should be compared with \cite[Proposition 8]{abr_aaa}. However we must observe that the Fr\'echet space ${\mathcal S}({\mathbb R})$ does not satisfy the hypothesis in the mentioned result.

\begin{corollary}\label{cor:spectra_mean-ergodic}{\rm If $C_\varphi:{\mathcal S}({\mathbb R})\to {\mathcal S}({\mathbb R})$ is mean ergodic then $\sigma(C_\varphi) \subset \overline{{\mathbb D}}.$
}
\end{corollary}
\begin{proof}
 Since $T = C_\varphi$ is mean ergodic then $\left(\frac{1}{n}T^n f\right)_n$ converges to zero for every $f\in {\mathcal S}({\mathbb R}).$ In particular, for every continuous seminorm $p$ on ${\mathcal S}({\mathbb R})$ and for every $f\in {\mathcal S}({\mathbb R})$ we have $p\left(T^n f\right) = O(n)$ as $n\to \infty.$ Now it suffices to apply Propositions \ref{prop:eigenvalues_disc} and \ref{prop:surjective}.
\end{proof}

\par\medskip
We recall that every power bounded operator on a Fr\'echet Montel space is mean ergodic.

\begin{corollary}\label{cor:eigen_power}{\rm If $C_\varphi:{\mathcal S}({\mathbb R})\to {\mathcal S}({\mathbb R})$ is power bounded and, for some $1 < p < \infty,$
$$
\sup_{x\in {\mathbb R}}\sum_{n=1}^\infty\frac{1}{\left(1 + \left|\varphi_n(x)\right|\right)^p} < \infty
$$ then $\sigma(C_\varphi) \subset {\mathbb D}.$
}
\end{corollary}
\begin{proof}
 Since $C_\varphi$ is power bounded, all the symbols $\left(\varphi_n\right)$ satisfy the symbol conditions in a uniform way (see Proposition \ref{prop:powerbounded}). That is, (i) there exist $k > 0$ such that $\left|\varphi_n(x)\right|^k \geq \left|x\right|$ whenever $\left|x\right|\geq k$ and (ii) for every $\ell\in {\mathbb N}$ there are constants $m_\ell$ and $C_\ell$ such that
$$
\left|\left(\varphi_n\right)^{(\ell)}(x)\right| \leq C_\ell \left(1+\left|\varphi_n(x)\right|\right)^{m_\ell} \ \ \forall x\in {\mathbb R}.
$$ From Corollary \ref{cor:spectra_mean-ergodic} we have $\sigma(C_\varphi) \subset \overline{{\mathbb D}}.$ Hence we only need to check that the operator $C_\varphi - \lambda I$ is an isomorphism whenever $|\lambda| = 1.$ In fact, the hypothesis imply that the series
$$
f:= -\sum_{n=0}^\infty\frac{1}{\lambda^{n+1}}g\left(\varphi_n(\cdot)\right)
$$ is convergent for every $g\in {\mathcal S}({\mathbb R})$ and $C_\varphi(f) - \lambda f = g.$ On the other hand, since all the orbits of $\varphi$ are unbounded we can apply Proposition \ref{prop:eigenvalues_disc} to get that $C_\varphi - \lambda I$ is also injective.
\end{proof}

The hypothesis of Corollary \ref{cor:eigen_power} are satisfied in the case that $\varphi(x) = \sqrt{x^2 + 1}.$ Combining Example \ref{ex:eigenvalues_sqrt} and Corollary \ref{cor:eigen_power} we get the following

\begin{example}{\rm Let $\varphi(x) = \sqrt{x^2 + 1}.$ Then
$$
\sigma(C_\varphi) = \sigma_p(C_\varphi) = {\mathbb D}.
$$
 }
\end{example}

The following example uses the Zak transform, a tool coming from time-frequency analysis that has shown to be very useful to characterize the spectrum of frame operators in some special cases. The Zak transform of a function $f\in L^1({\mathbb R})$ is defined by
$$
{\mathcal Z}f(x,\omega) = \sum_{k\in {\mathbb Z}}f(x-k)e^{2\pi i k\omega}.
$$ It turns out that ${\mathcal Z}f\in L^1([0,1]^2)$ and
\begin{equation}\label{eq:zak}
\int_0^1{\mathcal Z}f(x,\omega)e^{-2\pi i x\omega} dx = \widehat{f}(\omega)\ a.e\ \omega\in {\mathbb R}.
\end{equation}
If $f\in {\mathcal S}({\mathbb R})$ then ${\mathcal Z}f$ is a continuous funtion on ${\mathbb R}^2$ and the identity (\ref{eq:zak}) hods for every $\omega\in {\mathbb R}.$  See \cite[Section 8.1]{grochenig} for details.

\begin{example}\label{ex:translations}{\rm Let $\varphi(x) = x+1.$ Then $\sigma_p(C_\varphi) = \emptyset$ while
$$
\sigma(C_\varphi) = \left\{\lambda\in {\mathbb C}:\ \ |\lambda| = 1\right\}.
$$
}
\end{example}
\begin{proof}
 According to Proposition \ref{prop:spectra_increasing} we have $\sigma_p(C_\varphi) = \emptyset.$ Hence, we have to show that $C_\varphi - \lambda I$ is surjective if and only if $|\lambda|\neq 1.$
 \par\medskip
 (a) We first discuss the case $|\lambda|>1.$ For every $\ell\in {\mathbb N}$ we consider the continuous seminorm
 $$
 p_\ell(h):=\sup_{x\in {\mathbb R}}\sup_{0\leq k\leq \ell}(1+x^2)^\ell\left|h^{(k)}(x)\right|,\ \ h\in {\mathcal S}({\mathbb R}).
 $$ We will check that the hypothesis in Proposition \ref{prop:surjective} are satisfied. To this end, given $x\in {\mathbb R}$ and $n\in {\mathbb N}$ we consider the following cases:
 \par
 (i) $|x+n| \geq \frac{|x|}{2}.$ Then $(1+x^2)^\ell \leq 4^\ell\left(1 + (x+n)^2\right)^\ell$ and
 $$
 \left(1+x^2\right)^\ell\left|f^{(k)}(x+n)\right| \leq 4^\ell p_\ell(f) \ \ \forall\ 0\leq k\leq \ell.
 $$ \par
 (ii) $|x+n| < \frac{|x|}{2}.$ Then $x < 0$ and $|x|-n < \frac{|x|}{2},$ hence $|x| < 2n$ and
 $$
 \left(1+x^2\right)^\ell\left|f^{(k)}(x+n)\right| \leq \left(1+4n^2\right)^\ell p_\ell(f) \ \ \forall\ 0\leq k\leq \ell.
 $$ Consequently
  $$
 p_\ell\left(C_{\varphi_n} f\right) \leq \left(1+4n^2\right)^\ell p_\ell(f) \ \ \forall n\in {\mathbb N}.
 $$ We can apply Proposition \ref{prop:surjective} to conclude that $C_\varphi - \lambda I$ is surjective whenever $|\lambda| > 1.$
 \par\medskip
 (b) We now consider the case $|\lambda| < 1.$ We observe that $C_\varphi^{-1} = C_\psi$ where $\psi(x) = x-1.$ The same argument as in (a) shows that
 $$
 0 < |\lambda| < 1 \Rightarrow \frac{1}{\lambda}\notin \sigma(C_\psi) \Rightarrow \lambda \notin \sigma(C_\varphi).
 $$
 \par\medskip
 (c) We assume that $C_\varphi - \lambda I$ is surjective for some $\lambda = e^{2\pi i \omega},$ $\omega\in {\mathbb R}.$ We now fix $g\in {\mathcal S}({\mathbb R})$ and take $f\in {\mathcal S}({\mathbb R})$ such that $C_\varphi f = \lambda f + g.$ Then
 $$
 C_{\varphi_n} f = \lambda^n f + \sum_{k=0}^{n-1}\lambda^{n-1-k}C_\varphi^k g\ \ \forall\ n\in {\mathbb N}.
 $$ That is,
 $$
 f(x+n) = \lambda^n f(x) + \sum_{k=0}^{n-1}\lambda^{n-1-k}g(x+k)
 $$ or, equivalently,
 $$
 f(x) = \lambda^n f(x-n) + \sum_{k=1}^{n}\lambda^{k-1}g(x-k).
 $$ We take limits as $n\to \infty$ (with $x$ fixed) in the two previous identities and obtain
 $$
 f(x) = -\frac{1}{\lambda}\sum_{k=0}^\infty \frac{1}{\lambda^k}g(x+k)
 $$ and also
 $$
 f(x) = \frac{1}{\lambda}\sum_{k=1}^\infty\lambda^k g(x-k).
 $$ Consequently
 $$
 \sum_{k\in {\mathbb Z}}\lambda^k g(x-k) = 0\ \ \ \forall\ x\in {\mathbb R}.
 $$ This means that the Zak transform satisfies
 $$
 {\mathcal Z}g(x,\omega) = 0 \ \ \mbox{for every}\ x\in {\mathbb R}\ \mbox{and}\ g\in {\mathcal S}({\mathbb R}).
 $$ Since
 $$
 \int_0^1{\mathcal Z}g(x,\omega)e^{-2\pi i x\omega}\ dx = \widehat{g}(\omega)
 $$ we get a contradiction.
\end{proof}

\begin{example}{\rm Let $\varphi(x) = ax$ where $a\neq 0$ and $|a|\neq 1.$ Then $$\sigma(C_\varphi) = {\mathbb C}\setminus \left\{0\right\}.$$
 }
\end{example}
\begin{proof}
 Since $(C_\varphi)^{-1} = C_\psi$ for $\psi(x) = a^{-1}x$ then it suffices to consider the case $|a| > 1.$ We will show that $C_\varphi - \lambda I$ is not surjective for every $\lambda\neq 0.$ We distinguish two cases.
 \par\medskip
 (a) We first consider $|\lambda| \geq 1$ and assume that $f(ax) = \lambda f(x) + g(x),$ where $f,g\in {\mathcal S}({\mathbb R}).$ Then, an iterative argument gives
 $$
 f(a^n x) = \lambda^n \left(f(x) + \frac{1}{\lambda}\sum_{k=0}^{n-1}\frac{1}{\lambda^k}g(a^k x)\right).
 $$ For $x\neq 0$ we can take limits as $n\to \infty$ and obtain
 $$
 f(x) = - \frac{1}{\lambda}\sum_{k=0}^{\infty}\frac{1}{\lambda^k}g(a^k x).
 $$ From
 $$
 \left|g(a^k x)\right| \leq \frac{C}{1 + x^2 a^{2k}}
 $$ we deduce that the previous series converges absolutely and uniformly on $\left\{x\in {\mathbb R}:\ |x|\geq \varepsilon\right\}$ for every $\varepsilon > 0.$ A similar argument permits to conclude that the formal derivatives of the series converge absolutely and uniformly on $\left\{x\in {\mathbb R}:\ |x|\geq \varepsilon\right\}.$ Consequently, for every $x\neq 0$ and $j\in {\mathbb N},$
 $$
 f^{(j)}(x) = -\frac{1}{\lambda}\sum_{k=0}^{\infty}\left(\frac{a^j}{\lambda}\right)^k g^{(j)}(a^k x).
 $$ We now fix $j\in {\mathbb N}$ such that
 $$
 \left|\frac{a^j}{\lambda}\right| > 1.
 $$ We take a test function $\Phi$ such that $\Phi = 1$ on $[-1,1]$ and $\Phi(x) = 0$ for $|x|\geq |a|.$ Then
 $$
 g(x) = \frac{x^j}{j!}\Phi(x)
 $$ is a function in the Schwartz class such that $g^{(j)} = 1$ on $[-1,1]$ and $g^{(j)} = 0$ for $|x|\geq |a|.$ If $f\in {\mathcal S}({\mathbb R})$ satisfies $C_\varphi f = \lambda f + g$ then, for $x_m = a^{-m}$ we have
 $$
 f^{(j)}(x_m) = -\frac{1}{\lambda}\sum_{k=0}^m\left(\frac{a^j}{\lambda}\right)^k.
 $$ Hence $f^{(j)}$ is unbounded in any neighborhood of the origin, which is a contradiction.
 \par\medskip
 (b) Let $\lambda\neq 0,\ |\lambda| < 1$ and assume that $f(ax) = \lambda f(x) + g(x).$ Then, an iterative argument gives
 $$
 f(x) = \lambda^n f\left(\frac{x}{a^n}\right) + \frac{1}{\lambda}\sum_{k=1}^n\lambda^k g\left(\frac{x}{a^k}\right).
 $$ After taking limits as $n\to \infty$ we get
 $$
 f(x) = \frac{1}{\lambda}\sum_{k=1}^\infty\lambda^k g\left(\frac{x}{a^k}\right) \ \ \forall x\in {\mathbb R}.
 $$ We now take $g\in {\mathcal S}({\mathbb R})$ such that $g(x) = 1$ for $x\in [-1,1]$ and $g(x) = 0$ for $|x| \geq |a|.$ If $f\in {\mathcal S}({\mathbb R})$ satisfies $C_\varphi f = \lambda f + g$ then, for every $m\in {\mathbb N},$
 $$
 f(a^m) = \frac{1}{\lambda}\sum_{k=m}^\infty\lambda^k = \frac{\lambda^{m-1}}{1-\lambda}.
 $$ Take $j\in {\mathbb N}$ such that $\left|\lambda a^j\right| > 1.$ Then
 $$
 \left|a^m\right|^j\cdot \left|f(a^m)\right| \geq \frac{\left|\lambda a^j\right|^m}{\left|\lambda(1-\lambda)\right|},
 $$ which is unbounded as $m\to \infty.$ This is a contradiction with $f\in {\mathcal S}({\mathbb R}).$
\end{proof}
\par\medskip
In the case $\varphi(x) = -x$ we have $\sigma(C_\varphi)= \sigma_p\left(C_\varphi\right) = \left\{-1,1\right\}.$


\begin{thebibliography}{99}

\bibitem{abr} A. A. Albanese, J. Bonet and W. J. Ricker, \emph{Mean ergodic operators in Fr\'{e}chet spaces}, Ann.
Acad. Sci. Fenn. Math. \textbf{34} (2009), 401--436.

\bibitem{abr_aaa}A. A. Albanese, J. Bonet, and W. J. Ricker, \emph{Uniform convergence and spectra of operators in a class of Fr\'echet spaces}, Abstr. Appl. Anal. (2014), Art. ID 179027, 16 pp.


\bibitem{Bayart}{F. Bayart, E. Matheron; {\em Dinamics of linear operators}, Cambridge University Press, 2009.}

\bibitem{tesiMJ}
M.J. Beltr\'an, \emph{Operators on weighted spaces of holomorphic functions}, Thesis, 2014.

\bibitem{bgjj1} M. J. Beltr\'{a}n-Meneu, M. C. G\'{o}mez-Collado, E. Jord\'{a}, D. Jornet; {\em Mean ergodic composition operators on Banach spaces of holomorphic functions.}
 J. Funct. Anal. \textbf{270} (2016), 4369--4385.

\bibitem{bgjj2} M. J. Beltr\'{a}n-Meneu, M. C. G\'{o}mez-Collado, E. Jord\'{a}, D. Jornet;  {\em Mean ergodicity of weighted composition operators on spaces of holomorphic functions.} J. Math. Anal. Appl. \textbf{444} (2016), 1640--1651.

\bibitem{bd1} J. Bonet, P. Doma\'nski;  {\em A note on mean ergodic composition operators on spaces of holomorphic functions}. Rev. R. Acad. Cienc. Exactas Fís. Nat. Ser. A Math. RACSAM \textbf{105} (2011), 389--396.

\bibitem{bd2} J. Bonet, P. Doma\'nski;   {\em Power bounded composition operators on spaces of analytic functions.} Collect. Math. \textbf{62} (2011), 69--83.

\bibitem{bd3} J. Bonet, P. Doma\'nski;  {\em Hypercyclic composition operators on spaces of real analytic functions}. Math. Proc. Cambridge Philos. Soc. \textbf{153} (2012), 489--503.

\bibitem{br} J. Bonet and W. Ricker, \emph{Mean ergodicity of multiplication operators in weighted spaces of
holomorphic functions}, Arch. Math. \textbf{92} (2009), 428--437.

\bibitem{cgp} I. Chalendar, E.A. Gallardo-Gutiérrez, J.R. Partington; {\em Weighted composition operators on the Dirichlet space: boundedness and spectral properties}, Math. Ann. {\bf 363} (2015), 1265--1279.

\bibitem{cowen} C.C. Cowen, B.D. MacCluer; {\em Composition operators on spaces of
analytic functions.} Studies in Advanced Mathematics. CRC Press,
Boca Raton, FL, 1995.

\bibitem{FLW} V. P. Fonf , M. Lin and P.  Wojtaszczyk;  \emph{Ergodic characterizations of reflexivity in Banach spaces}, J. Funct.
Anal. \textbf{187} (2001), 146--162.

\bibitem{GJ18} A. Galbis, E. Jord\'a; {\em Composition operators on the Schwartz space}, to appear in Rev. Mat. Iberoam., arXiv:1511.03072v1 [math.FA]


\bibitem{gs} E.A. Gallardo-Guti\'errez, R. Schroderus; {\em The spectra of linear fractional composition operators on weighted Dirichlet spaces},  J. Funct. Anal. {\bf 271} (2016), 20--745.

\bibitem{golinski} M. Golinski
{\em Operator on the space of rapidly decreasing functions with all non-zero vectors hypercyclic}, Adv. Math. {\bf 244} (2013), 663--677.

\bibitem{grochenig} K. Gr\"{o}chenig; Foundations of Time-Frequency
Analysis, {\em Birkh\"{a}user} (2001).

\bibitem{hlns} O. Hyvärinen, M. Lindstr\"{o}m, I. Nieminen, E. Saukko; {\em Spectra of weighted composition operators with automorphic symbols}, J. Funct. Anal. {\bf 265} (2013), 1749--1777.

\bibitem{GE_Peris}K. G.  Grosse-Erdmann and A.  Peris; Linear Chaos. Springer, Berlin, 2011.


\bibitem{kw} N. Kenessey, J. Wengenroth; {\em Composition operators with closed range for smooth injective symbols ${\mathbb R}\to {\mathbb R}^d$.} J. Funct. Anal. {\bf 260} (2011), 2997--3006.

\bibitem{lorch} E. R. Lorch;
{\em Means of iterated transformations in reflexive vector spaces.}
Bull. Amer. Math. Soc. \textbf{45} (1939), 945--947.

\bibitem{adam1} A. Przestacki; {\em Composition operators with closed range for one-dimensional smooth symbols.} J. Math. Anal. Appl. {\bf 399} (2013), 225--228.

\bibitem{adam2} A. Przestacki; {\em Characterization of composition operators with closed range for one-dimensional smooth symbols.} J. Funct. Anal. {\bf 266} (2014), 5847--5857.


\bibitem{adam2.5} A. Przestacki {\em Corrigendum to "Characterization of composition operators with closed range for one-dimensional smooth symbols''} [J. Funct. Anal. \textbf{266} (2014) 5847--5857][MR3182962]. J. Funct. Anal. \textbf{269} (2015), 2665--2667.

\bibitem{adam3}  A. Przestacki; {\em Dynamical properties of weighted composition operators on the space of smooth functions.} J. Math. Anal. Appl. \textbf{445} (2017), 1097--1113.


\bibitem{shapiro} J. H. Shapiro; {\em Composition operators and classical function theory.} Universitext: Tracts in Mathematics. Springer-Verlag, New York, 1993.




\end{thebibliography}
\end{document}